\documentclass{article}
\usepackage{graphicx}
\usepackage[utf8]{inputenc}

\title{Fried's theorem for  boundary geometries of rank one symmetric spaces}
\author{Raphaël V. {\sc Alexandre}\footnote{Institut de Math\'ematiques de Jussieu-Paris Rive Gauche, Sorbonne Université, 4 Place Jussieu, 75252 Paris Cédex, France. Email address: {\tt raphael.alexandre@imj-prg.fr}}}

\usepackage{amsmath, amsthm, amssymb, amsfonts}

\newtheorem{theorem}{Theorem}[section]
\newtheorem{lemma}[theorem]{Lemma}
\newtheorem{proposition}[theorem]{Proposition}
\newtheorem*{proposition*}{Proposition}
\newtheorem*{theorem*}{Theorem}
\newtheorem{corollary}[theorem]{Corollary}

\usepackage[backend=bibtex,style=alphabetic]{biblatex}
\bibliography{ref.bib}

\begin{document}
\maketitle
\begin{abstract}
After introducing the different boundary geometries of rank one symmetric spaces,  we state and prove Fried's theorem in the general setting of all those geometries: a closed manifold with a similarity structure  is either complete or the developing map is a covering onto the Heisenberg-type space deprived of a point.
\end{abstract}

\section{Introduction}

Let $\mathbf{F}$ be the field of the real, complex, quaternionic or octonionic numbers. We are interested in the boundary geometries $({\rm PU}(n,1;\mathbf{F}),\partial \mathbf{H}^n_\mathbf{F})$. Those structures will be considered for $n\geq 2$. For the non-real case, the  hyperbolic lines (when $n=1$) are isometric to real hyperbolic spaces, so the hypothesis $n\geq2$ is  only  a convenience. When $\mathbf{F}$ is the octonionic field,  we  only consider the  case $n=2$.
For example, when $\mathbf{F}=\mathbf{R}$, we get the flat conformal structure. When $\mathbf{F}=\mathbf{C}$, we get the spherical CR structure. 

In this paper we will prove the following theorem, called Fried's theorem since it was stated and proved by Fried for the real case  in \cite{Fried}. It was also proved for the complex case  by Miner \cite{Miner} and for the quaternionic case by Kamishima \cite{Kamishima}. It seems that the octonionic case has not been proved yet.

Start with a rank one symmetric space.
We set $\mathcal N$ the geodesic boundary of the space deprived of a point and ${\rm Sim}(\mathcal N)$ the subgroup of the isometries stabilizing this point. 

\begin{theorem*}[\ref{thm-fried}]
Let $M$ be a closed $({\rm Sim}(\mathcal{N}),\mathcal{N})$-manifold. If the developing map $D:\tilde M\to \mathcal{N}$ is not a cover onto $\mathcal{N}$, then the holonomy subgroup $\Gamma$ fixes a point in $\mathcal{N}$ and $D$ is in fact a covering onto the complement of this point.
\end{theorem*}

The proof proposed here of Fried's theorem \ref{thm-fried} will simultaneous deal with all the cases. This unified proof will  use a general approach of convexity in $\mathcal N$. Convexity arguments were crucial in Fried's initial approach. In the real case, the space and the tangent space on one hand, and the geodesic structure and the algebraic sum on the other hand, are each time essentially the same. Therefore a clarification was required to state general convexity arguments for every field.

A consequence of Fried's theorem is the following result. We keep the notations of the preceding theorem and denote by $L(\Gamma)$ the limit set of the holonomy group $\Gamma$. 

\begin{theorem*}[\ref{thm}]
Let $M$ be a closed $({\rm PU}(n,1;\mathbf F),\partial \mathbf{H}^n_{\mathbf F})$-manifold. 
If $D$ is not surjective then it is a covering onto its image. Furthermore,
 $D$ is a covering on its image if, and only if, $D(\tilde M)$ is equal to a connected component of $\partial \mathbf{H}^n_{\mathbf F}-L(\Gamma)$.
\end{theorem*}

The first section being the present introduction, the second section will introduce the different notions required.
Notations for $(G,X)$-structures will be introduced in \ref{sec:complete}. The third section consists of the proof of Fried's theorem. The forth consists of the proof of theorem \ref{thm}.

\tableofcontents

\section{Hyperbolic geometry}

For the moment we will set the octonionic case aside.
Let $\mathbf{F}$ be the field of the real or complex or quaternionic numbers. We denote by $\mathbf{F}^{n,1}$ the space $\mathbf{F}^{n+1}$ endowed with the quadratic form 
\begin{equation}
Q_E(z_1,\dots,z_{n+1}) = z_1z_1^* + \dots + z_{n}z_{n}^* - z_{n+1}z_{n+1}^* ,
\end{equation}
where $z^*$ denotes the conjugated of $z$. We specify $E$ in $Q_E$ to indicate that $Q$ is given in the canonical basis $(e_1,\dots,e_{n+1})$ of $\mathbf{F}^{n+1}$.

Now, let $F$ be the basis
\begin{equation}
f_1 = \frac{-e_1 + e_{n+1}}{\sqrt 2}, f_2 = e_2 , \dots, f_{n}=e_n, f_{n+1} = \frac{e_1+e_{n+1}}{\sqrt 2}.
\end{equation}
It is clear that it is an orthonormal basis with respect to $Q$. In that basis, $Q_F$ is given by
\begin{equation}
 Q_F(w_1,\dots,w_{n+1}) = w_2w_2^* + \dots+ w_nw_n^* - 2{\rm Re}(w_1w_{n+1}^*).
\end{equation}

Here, we denote by ${\rm Re}(p)$ and ${\rm Im}(p)$ the real and imaginary parts of $p$ such that $p = {\rm Re}(p)+{\rm Im}(p)$ and ${\rm Re}(p) \in \mathbf{R}$ and ${\rm Im}(p)\in {\rm Im}(\mathbf{F})={\rm Re}^{-1}(0)$.

We denote by ${\rm U}(k)$ the group of the unitary matrices of $\mathbf{F}^k$.
We denote by ${\rm U}(n,1)$ the group of the unitary matrices of $\mathbf{F}^{n,1}$, \emph{i.e.} the matrices acting by isometries on $\mathbf{F}^{n,1}$. We denote by $\mathbf{H}^n$ the projectivized of the negative subspace $\{Q(p)<0\}$ and by $\partial \mathbf{H}^n$ the projectivized null subspace $\{Q(p)=0\}$. Of course topologically speaking, $\partial \mathbf{H}^n$ is in fact the boundary of $\mathbf{H}^n$. The image of ${\rm U}(n,1)$ in ${\rm PGL}(\mathbf{F}^{n+1})$ is denoted by ${\rm PU}(n,1)$.

\par\vspace{\baselineskip}\noindent
We recover the ball-model of $\mathbf{H}^n$ as follows. In $\mathbf{P}(\mathbf{F}^{n+1})$ we take the affine chart $z_{n+1}=1$. Thus,
\begin{equation}
Q_E(z) \leq 0 \iff |z_1|^2 +\dots + |z_n|^2 \leq 1.
\end{equation}
In what follows we will take these coordinates of $\mathbf{H}^n$.
In $\mathbf{H}^n$, the vector $f_1$ represents “$-1$”$ = [-1,0,\dots,0,1]$ and $f_{n+1}$ represents “$1$”$=[1,0,\dots,0,1]$ which are  in $\partial \mathbf{H}^n$. In the basis $E$, the vector $e_{n+1}$ represents “$0$”$=[0,\dots,0,1]$ and is in $\mathbf{H}^n$. The other vectors of the basis $E$ and $F$ are not in $\mathbf{H}^n$ nor in $\partial \mathbf{H}^n$.

\subsection{The $KAN$-Iwasawa decomposition}

We  will now always suppose that $n\geq 2$. 

We denote by $K$ the subgroup of ${\rm PU}(n,1)$ given by  the matrices in the canonical basis
\begin{equation}
k = 
\begin{pmatrix}
k'\\& 1
\end{pmatrix}
\end{equation}
where $k'$ belongs to ${\rm U}(n)$.

We denote by ${\rm U}(\mathcal N)$ or also $M$ the subgroup of $K$ stabilizing $\mathbf F f_1$ and $\mathbf F f_{n+1}$. 
In general, see \cite{Kim}, the matrices of $M$ are given in the basis $F$ by 
\begin{equation}
m = 
\begin{pmatrix}
\alpha \\ &\alpha m' \\ && \alpha
\end{pmatrix}
\end{equation}
with $m'\in {\rm U}(n-1)$ and $\alpha\in {\rm U}(1)$. Since we projectivize on the right, the unitary factor $\alpha$ can only be forgotten in the cases $\mathbf F=\mathbf R$ or $\mathbf F=\mathbf C$.

We denote by $A$ the subgroup of ${\rm PU}(n,1)$ given by the matrices in the basis $F$
\begin{equation}
a_t = 
\begin{pmatrix}
e^{-t} \\ & E_{n-1} \\ && e^t
\end{pmatrix}
\end{equation}
where $t\in \mathbf{R}$ and $E_{n-1}$ is the identity matrix.
We denote by $N$ the subgroup of ${\rm PU}(n,1)$ given by the matrices in the basis $F$
\begin{equation}
n_{u,I} = 
\begin{pmatrix}
1 \\u & E_{n-1}\\\|u\|^2/2 + I & {^t}u^* & 1
\end{pmatrix}
\end{equation}
where $(u,I)\in \mathbf{F}^{n-1}\times{\rm Im}(\mathbf{F})$. The vector $u$ is to be thought as in $(1,u,1)$ in $\mathbf{F}^{n,1}$.

We denote by $P$ the product $MAN$.
Classical linear algebra shows the following lemma.
\begin{lemma}\label{kan}
The following properties are true.
\begin{enumerate}
\item The action of $N$ is free and transitive on $\partial \mathbf{H}^n-\{1\}$.
\item The subgroup $K$ is the stabilizer of $0$ and is transitive on $\partial \mathbf{H}^n$.
\item We have ${\rm PU}(n,1)=KP=KAN$.
\item The subgroup $P$ is the stabilizer of $1$ and is transitive on $\mathbf{H}^n$.
\end{enumerate}
\end{lemma}

\subsection{The group $N$}

We will now closely study  the subgroup $N\subset {\rm PU}(n,1)$. First, we identify $n_{u,I}$ with the couple $(u,I)\in \mathbf{F}^{n-1}\times {\rm Im}(\mathbf{F})$.
We have:
\begin{align}
(u,I)+(v,J) &:= n_{u,I}n_{v,J} \nonumber \\
  &=
  \begin{pmatrix}
  1 \\
  u+v & E_{n-1} \\
  \|u\|^2/2 + I + u^*v + \|v\|^2/2 + J  & {^t}(u+v)^* &1
  \end{pmatrix} \nonumber \\
  &=
  \begin{pmatrix}
  1 \\
  u+v & E_{n-1}   \\
  \|u+v\|^2/2+{\rm Im}(u^*v) +I+J & {^t}(u+v)^* & 1
  \end{pmatrix}\nonumber\\
  &= (u+v, I+J + {\rm Im}(u^*v)).
\end{align}
Indeed, computations show that
\begin{align}
{\rm Im}(u){\rm Im}(v) &= {\rm Re}(u){\rm Re}(v) + {\rm Re}(u){\rm Im}(v) - {\rm Re}(v){\rm Im}(u) - u^*v \\
  {\rm Im}(u){\rm Im}(v) + {\rm Im}(v){\rm Im}(u) &= 2{\rm Re}(u){\rm Re}(v) - 2{\rm Im}(u^*v)\\
\|u+v\|^2   &= {\rm Re}(u)^2 + 2 {\rm Re}(u){\rm Re}(v) - {\rm Im}(u)^2 - {\rm Im}(u){\rm Im}(v)\nonumber  \\
& \hphantom{=}+{\rm Re}(v)^2 - {\rm Im}(v){\rm Im}(u) - {\rm Im}(v)^2  \\
&= \|u\|^2 + \|v\|^2 + 2 {\rm Re}(u^*v)\\
\|u\|^2/2 +  \|v\|^2/2 + u^*v &= \|u+v\|^2/2 + {\rm Im}(u^*v).
\end{align}

This addition is not commutative, since ${\rm Im}(u^*v)$ is not symmetric. We have $-(u,I)=(-u,-I)$. We also have additivity since
\begin{align}
{\rm Im}(u^*v) + {\rm Im}((u+v)^*w) &= {\rm Im}(u^*v) + {\rm Im}(u^*w) + {\rm Im}(v^*w) \\ 
&= {\rm Im}(u^*(v+w))+{\rm Im}(v^*w).
\end{align}

By the action of $P$ on $\partial \mathbf{H}^n$, the point $1$ is fixed. Also, the action of $N$ on the point $-1$ is free and transitive on $\partial \mathbf{H}^n-\{1\}$ (see lemma \ref{kan}).  We denote by $\mathcal{N}$ the  image of $-1$ by the correspondance $n_{u,s}\leftrightarrow (u,s)$.
Thus, we geometrically obtain  $\mathcal{N}=\partial \mathbf{H}^n-\{1\}$ and also $\mathcal{N}$ isomorphic to $N$. Note that, $-1\in \partial \mathbf{H}^n$ corresponds to $0\in \mathcal{N}$ and $1\in \partial \mathbf{H}^n$ to $\infty\in \mathcal{N}$.

Now, we look at the action of $M$ and $A$ on $N$ by looking of the action  on $-1$. 
Easy calculations show that
\begin{align}
mn_{u,I}(-1) &= n_{\alpha m'u \alpha^{-1},\alpha I\alpha^{-1}}(-1),\\
a_tn_{u,I}(-1) &= n_{e^tu,e^{2t}I}(-1).
\end{align}
Thus, we can see that the correspondance $n_{u,I}\leftrightarrow (u,I)$ gives $a_tn_{u,I} \leftrightarrow (e^t u ,e^{2t} I)$ and $mn_{u,I} \leftrightarrow (\alpha m'u\alpha^{-1},\alpha I\alpha^{-1})$. Changing $e^t$ for $\lambda$ and $(\alpha m',\alpha)$ for $(P,\alpha)$, denoting $\lambda:= a_t$ and $P:=m$, we get the actions :
\begin{align}
P(u,I) &= (Pu\alpha^{-1},\alpha I \alpha^{-1})\\
\lambda (u,I) &= (\lambda u, \lambda^2 I).
\end{align}
Remark that the actions of $A$ and $M$ commute.
If $f\in {\rm Sim}(\mathcal{N})$, then we can express $f$ as
\begin{equation}
f(x) = \lambda P(x) + c
\end{equation}
with $\lambda \in \mathbf{R}_+$, $P\in {\rm U}(\mathcal N)$ and $c\in \mathcal{N}$.

\par\vspace{\baselineskip}\noindent
Now we give a pseudo-norm on $\mathcal{N}$ (compare with \cite[p. 10]{Cowling}). For $x = (u,I)\in \mathcal{N}$ we define 
\begin{equation}
 \|x\|_\mathcal{N}^2 := \|u\|_{\mathbf{F}^{n-2}}^2 + \|I\|_\mathbf{F}. 
\end{equation}
It is easy to see that $\|x\|=0$ if and only if $x=(0,0)$ and that
\begin{equation}
 \|\lambda x\|_\mathcal{N} = \lambda \|x\|_\mathcal{N}
\end{equation}
for any $\lambda \in \mathbf{R}_+$. The triangle inequality is respected.

To $\|\cdot\|_\mathcal{N}$ we can associate the distance function 
\begin{equation}
d_\mathcal{N}(x,y) = \|-x+y\| = d_{\mathcal N}(0,-x+y).
\end{equation}
and remark that it is symmetric since $\|-(u,I)\|=\|(-u,-I)\|=\|(u,I)\|$ and $\|-y+x\|=\|-(-x+y)\|$.

\subsection{The octonionic case}

The special case $\mathbf F = \mathbf O$ needs extra care (see \cite{Baez} for a general survey on the octonions). Indeed, $\mathbf O^3$ is not a vector space and the preceding construction relies on this structure. However, as explained by Allcock in \cite{Allcock} (compare also with \cite{Harvey} for the projective plane), most of the construction can still be made.

We  can still speak of the hyperbolic plane $\mathbf H^2_{\mathbf O}$ as a smooth manifold. It can even be made from elements of $\mathbf O^3$ such that the coordinates lie in an associative algebra \cite[p. 12]{Allcock} (we have different conventions).

Again, ${\rm PU}(2,1;\mathbf O) := {\rm Aut}(\mathbf H^2_{\mathbf O})$ acts transitively on the boundary at infinity (which can still be defined) $\partial \mathbf H^2_{\mathbf O}$. The $15$-dimensional subgroup N  stabilizing a point $\infty \in\partial \mathbf H^2_{\mathbf O}$ is similar to the Heisenberg group (it is a $H$-type group -- see \cite{Cowling} for an algebraic approach), hence giving the space $\mathcal N$ by the same identification. Translations are again given by
\begin{equation}
(x,z) + (\xi,\eta) = (x+\xi, z+\eta + {\rm Im}(x^*\eta)).
\end{equation}
Here, the coordinates $(x,z)$ are as previously with $x\in\mathbf O$ and $z\in {\rm Im}(\mathbf O)$. For the same reason as for the other fields, the addition of $\mathcal N$ remains associative since the addition in $\mathbf O$ is associative, and the addition in $\mathcal N$ is not commutative.
As in the general case, the commutator of two translations is given by $(0,2{\rm Im}(x^*\eta))$.
If $\mu$ is a unit imaginary, then we can define a rotation by $\mu$ on $\mathcal N$.
\begin{equation}
m_\mu (x,z) = (\mu x, \mu z \mu^{-1}).
\end{equation}
It gives a compact group of rotations.
If $\lambda\in \mathbf R_+$, then again we have dilatations given by
\begin{equation}
\lambda (x,z) = (\lambda x, \lambda^2 z).
\end{equation}
So the group ${\rm Sim}(\mathcal N)$ is very similar to the previous cases.
We also remark that the stabilizer $K$ of a point in the interior $0\in \mathbf H^2_{\mathbf O}$ is still a compact group and again $ \mathbf H^2_{\mathbf O} = {\rm Aut}( \mathbf H^2_{\mathbf O})/K$. Finally, the same norm $\|\cdot\|_{\mathcal N}$ can be defined and used as previously.

More references can be given. In \cite{Riehm1}, Riehm shows that the isometry group of $\partial \mathbf H^2_{\mathbf O}$ stabilizing $\infty$ is transitive on the unit sphere. In \cite{Riehm2}, Riehm shows that every geodesic is a one-parameter group of global isometries. Those results and  a full description of the automorphism group are described in \cite{Cowling}. Note that those results are common for every field $\mathbf F$ we consider here.

\subsection{Limit sets}

In what follows we will use limit sets (compare with \cite{Chen}). Given a subgroup $\Gamma\subset {\rm PU}(n,1)$, we can define the \emph{limit set of $\Gamma$}, $L(\Gamma)$, as $\overline{\Gamma\cdot p}\cap \partial \mathbf{H}^n$ for any $p\in \mathbf{H}^n$.
This does not depend on the choice of $p$ since for $q\in \mathbf{H}^n$, the distance between $p$ and $q$ is preserved by elements of $\Gamma$. Hence, if $g_np\to x\in \partial \mathbf{H}^n$, then so must $g_nq$.

Also, for $x\in L(\Gamma)$ such as $x=\lim g_np$,  any other point $z$ of $\partial \mathbf{H}^n$ except possibly one must verify $\lim g_n z = x$. Indeed, given $z_1,z_2$  different from one another and both distinct from $x$, then for $\gamma$ a  geodesic between $z_1,z_2$, since the points of the geodesic go to $x$, then so must one extremity of the geodesic. If $y$ is the only point that does not tend to $x$, then we call $(x,y)$ a \emph{dual pair}, or \emph{dual points}. In this case, it is easy to see that $\lim g_n^{-1} p =y$.

The limit set $L(\Gamma)$ is stable by the action of $\Gamma$. Moreover, it is the minimal invariant set by $\Gamma$: if $A\subset \partial \mathbf{H}^n$ is \emph{closed and invariant by $\Gamma$ and is at least constituted of two points}, then $L(\Gamma)\subset A$. This fact can be deduced from the preceding remark.

The following lemma will be used in what follows. The proof is easily extended to the case where $\mathbf{F}$ is the octonion field. It uses the fact that $K$ in the decomposition $KAN$ is maximal compact.
Another way to deduce this lemma is given by CAT(0) theory (see \cite[p. 179]{Bridson}).

\begin{lemma}[See {\protect\cite[p. 79]{Chen}}]\label{cg-lem-fixpt}Let $\Gamma\subset {\rm PU}(n,1)$ be a subgroup.
Suppose $L(\Gamma)=\emptyset$, then the elements of $\Gamma$ let a point  fixed in $\mathbf{H}^n$.
\end{lemma}

\subsection{Complete structures}\label{sec:complete}

A $(G,X)$-structure (compare with Thurston's textbook \cite{Thurston}), is a pair of a smooth space $X$ with a transitive group $G$ acting by analytic diffeomorphisms. (A more general notion in \cite{Thurston} appears by allowing diffeomorphisms to be only locally defined, but we won't need this generality here.) A smooth manifold $M$ gets a $(G,X)$-structure if we can choose an atlas of $M$ consisting of charts defined on open sets of $X$ and with transition maps belonging to $G$. In this case we speak of a $(G,X)$-manifold to designate a manifold together with a $(G,X)$-structure.

Such a $(G,X)$-structure on a smooth manifold $M$ gives a pair $(D,\rho)$ of the \emph{developing map} $D:\tilde M \to X$ and the \emph{holonomy map} $\rho:\pi_1(M)\to G$. The developing map is a local diffeomorphism. The image of the holonomy map is called the \emph{holonomy group}, and is usually denoted here by $\Gamma=\rho(\pi_1(M))$. This pair $(D,\rho)$ prescribes the $(G,X)$-structure on $M$.

Two general problems naturally arise: it is hard to say when it is possible to put a $(G,X)$-structure on $M$ (it is the \emph{geometrization} problem); and there are very few general properties on a pair $(D,\rho)$. Two properties on $(D,\rho)$ are interesting to investigate: the completeness ($D$ is a covering), and the discreteness ($\Gamma$ is discrete in $G$).

We will say that the structure is \emph{complete} if the developing map $D:\tilde M\to X$ is a covering onto $X$.  Of course, complete $({\rm PU}(n,1),\partial\mathbf H^n)$-manifolds are rare (those have a spherical structure), and we will rather ask if $D:\tilde M \to X$ is a covering onto its image (but we keep the term \emph{complete} for a covering onto $X$). Remark that if $D$ is a covering map onto a simply connected space, then $\Gamma$ is discrete. (The converse is not true in general, see for example \cite{Falbel8}.)

The complete structures that we will encounter come from the following lemma. None of these results are new. (Compare again with \cite{Thurston}.)

\begin{lemma}\label{lem-comp-met}
Let $(G,X)$ be a geometrical structure such that $X$ has a $G$-invariant riemannian metric. If $M$ is a closed $(G,X)$-manifold, then it is complete.
\end{lemma}
\begin{proof}
Let $D:\tilde M \to X$ be the developing map. It is a local diffeomorphism. Pulling-back the metric on $\tilde M$, it makes $D$ a local isometry. Since $M$ is closed, $D$ is a covering map.
\end{proof}

Therefore, to show that a $(G,X)$-structure is complete for a closed manifold $M$, it suffices to show that $G$ has compact stabilizers  on $X$ (compare with \cite[p. 144]{Thurston}).

Recall that $K\subset {\rm PU}(n,1)$ from the Iwasawa decomposition is compact. And remark that $({\rm Sim}(\mathcal{N})_0,\mathcal{N}-\{0\})$ has compact stabilizers.

\begin{lemma}\label{cov-lem-k}\label{lem-n-com}
Any closed $(K, \partial \mathbf{H}^n)$ or $({\rm Sim}(\mathcal{N})_0,\mathcal{N}-\{0\})$-manifold is complete.\qed
\end{lemma}

\begin{lemma}\label{lem-gamma-eq}Let $M$ be a closed $({\rm PU}(n,1),\partial\mathbf H^n)$-manifold. Let $\Gamma$ be the holonomy group.
If $L(\Gamma)=\emptyset$, then $D:\tilde M \to \partial \mathbf H^n$ is  a covering map (therefore a diffeomorphism).
\end{lemma}

\begin{proof}
If $L(\Gamma)=\emptyset$, then the elements of $\Gamma$ let a point fixed in $\mathbf{H}^n$ by lemma \ref{cg-lem-fixpt}. Up to conjugation it is $0$, hence $\Gamma\subset K$. By lemma \ref{cov-lem-k}, it follows that $D$ is a covering map.
\end{proof}

\begin{corollary}Let $M$ be a closed $({\rm PU}(n,1),\partial \mathbf{H}^n)$-manifold.
If $L(\Gamma)=\emptyset$, then $M$ is a spherical manifold.
In particular, if $M$ is a simply connected closed $({\rm PU}(n,1),\partial \mathbf{H}^n)$-manifold, then $M$ is  diffeomorphic to a  sphere.
\end{corollary}
\begin{proof}
Because $K\simeq {\rm U}(n)\subset {\rm O}(kn)$ (where $k=1,2,4,8$ depending on the field) and $\partial \mathbf{H}^n$ is a sphere of real dimension ($kn-1$).
\end{proof}

The following proposition is also known as the “cutting lemma” in a paper of Falbel and Gusevskii \cite[th. 2.3]{Falbel}, in which they directly refer  to a paper of Kulkarni and Pinkall \cite[th. 4.2]{Kulkarni}. A proof can be found in \cite{Kulkarni}. This result will be of great use in the proofs of Fried's theorem \ref{thm-fried} and theorem \ref{thm}.

\begin{proposition}\label{prop-deuxpoints}
Let $M$ be a closed $({\rm PU}(n,1),\partial\mathbf{H}^n)$-manifold. Denote $\Gamma$ the holonomy group and $\Omega=\partial\mathbf{H}^n-D(\tilde M)$.
Suppose that $L(\Gamma)\subset \Omega$ and that $L(\Gamma)$ has at least two points.
Then $D$ is a covering map onto a connected component of $\partial \mathbf{H}^n-L(\Gamma)$.
\end{proposition}

\section{Fried's theorem and similarity structures}

The goal of this section is to show the following theorem, initially stated and shown by Fried for the real case $\mathbf{F}=\mathbf{R}$ in \cite{Fried}. The complex case $\mathbf F=\mathbf C$ was proved by Miner in \cite{Miner}.  Later, a  different proof (but still in the initial ideas of Fried) of the real case was given by Matsumoto in his survey \cite{Matsumoto}. An analytic proof of the real case was done by \cite{Bridson}. The quaternionic case $\mathbf F=\mathbf H$ was addressed by Kamishima in \cite{Kamishima}.
It seems that the octonionic case $\mathbf F=\mathbf O$ has  not been proved yet.
We will simultaneously prove the theorem for all the fields $\mathbf{F}$ considered.

\begin{theorem}\label{thm-fried}
Let $M$ be a closed $({\rm Sim}(\mathcal{N}),\mathcal{N})$-manifold. If the developing map $D:\tilde M\to \mathcal{N}$ is not a cover onto $\mathcal{N}$, then the holonomy subgroup $\Gamma$ fixes a point in $\mathcal{N}$ and $D$ is in fact a covering onto the complement of this point.
\end{theorem}

The different ideas of the proof  come from  \cite{Fried}, \cite{Miner} and \cite{Matsumoto}. The ideas about convexity of Fried and Miner (also in Carrière's work \cite{Carriere}) forged the necessity of the second section  and the ideas for the first arguments of the 
 theorem's proof. Matsumoto's  ideas helped to find the last arguments.

\subsection{A shortcut: discrete holonomy and autosimilarity}

A consequence of Fried's theorem is that \emph{for any similarity structure on a closed manifold, the holonomy is discrete}. The converse also holds.

\begin{proposition}\label{prop-j-group}
Suppose that $\Gamma\subset{\rm Sim}(\mathcal{N})$ is a discrete subgroup. Then either $L(\Gamma)=\emptyset$, or $L(\Gamma)=\{\infty\}$ or $L(\Gamma)=\{\infty,a\}$ for some $a\in \mathcal{N}$.
\end{proposition}

This can be compared with Matsumoto \cite[lemma 4.20]{Matsumoto}.

\begin{proof}
Suppose that $L(\Gamma)$ is neither $\emptyset$ nor $\{\infty\}$. Since $\Gamma\subset {\rm Sim}(\mathcal{N})$, we have $\infty\in L(\Gamma)$. Indeed, if not, then $\Gamma\subset M$ but $L(M)=\emptyset$.

Now, let $f,g$ be such that $f^n(x)\to a$ and $g^n(x)\to b$. This hypothesis can be made because if $L(\Gamma)\neq \{\infty\}$ then there exist $f,g$ with dilatation factors different from $1$. If $a=b$ for all choices $f,g$ then $\Gamma$ fixes $a$ and it follows that $L(\Gamma)=\{\infty\}$ or $L(\Gamma)=\{a,\infty\}$. So by absurd, suppose $a\neq b$.

Denote $f(x)=\lambda P(x)+c$ and $g(x) = \mu Q(x)+d$.
Take $h=g\circ f \circ g^{-1}$, we see that $h$ fixes $g(a)$. Set also the sequence $h_n = f^n\circ h \circ f^{-n}$. The fixed point of $h_n$ is $f^n(g(a))$ tending to $a$.
\begin{align}
g\circ f\circ g^{-1}(x) &= \mu Q\left(\lambda P(\mu^{-1}Q^{-1}(x) +d' ) + c\right) + d\nonumber \\
&= \lambda QPQ^{-1}(x) + e\\
f^n\circ h\circ f^{-n}(x) &= \lambda P^nQPQ^{-1}P^{-n} (x) + c_n
\end{align}
Denote $A_n=P^nQPQ^{-1}P^{-n}$. Since $A_n\in M$ which is compact, we can extract a subsequence converging to $A$. Now, the constant $c_n$ must tend to $c$ since the fixed point tends to $a$. Hence, $h_n(x)$ accumulates to $\lambda A(x) + c$, contradicting the discreteness of $\Gamma$.
\end{proof}

From this property it is not difficult to retrieve Fried's theorem fixed point property. If $M$ is not complete, then $L(\Gamma)$ is neither $\emptyset$ nor $\{\infty\}$. Indeed, there must be an element of $\Gamma$ with dilatation factor different from $1$, giving a limit point in $\mathcal N$ (otherwise the structure is complete, since there would exist an invariant riemannian metric). Therefore by discreteness, $L(\Gamma)$ is $\{a,\infty\}$ and $a$ must be fixed.
For closed manifolds, it is then possible to retrieve the full Fried's theorem with the fixed point property.

Of course, the hypothesis that $\Gamma$ is discrete is highly non trivial, and this is why the  proof of Fried's theorem  is important. However, the author wishes to emphasize the fact that it remains true for any discrete subgroup (even if the manifold considered is not compact). In particular, it is possible to prove that any group with the property that \emph{if a point of $L(\Gamma)$ is totally fixed by $\Gamma$ then $L(\Gamma)$ is elementary (i.e. has at most two points)} enables to prove that $L(\Gamma)$ is autosimilar when it is not elementary. For take $p\in L(\Gamma)$ and $U$ an open neighborhood of $p$. The complement $\partial\mathbf{H}^n- \Gamma\cdot U$ is closed and invariant. Hence it must be at most a single point totally fixed by $\Gamma$ since $L(\Gamma)$ contains $p$ and would be contained in this complement if there were more than one point. Now $L(\Gamma)$  can in fact not have a totally fixed point in $\partial\mathbf{H}^n- \Gamma\cdot U$ since we supposed that $L(\Gamma)$ is infinite and verifies the property emphazed before. Therefore $\Gamma\cdot U$ covers $L(\Gamma)$, and this shows the autosimilarity property.

\subsection{The geometry of $\mathcal{N}$}

Some facts about the geometry of $\mathcal{N}$ will be needed. In the real case, $\mathcal{N}$ is the Euclidean space endowed with its similarities. The advantage is that the Euclidean space $\mathbf{R}^n$ is flat and we are allowed to state $\exp_x(v)=x+v$.

It is still true in general.  First, $\mathcal N$ is a 2-nilpotent Lie group. It follows that the Lie algebra $\mathfrak n$ of $\mathcal{N}$ is 2-nilpotent and the exponential map is a diffeomorphism between $\mathfrak n$ and $\mathcal{N}$. The real vector space $\mathfrak n$ is is to be thought as a global coordinate system of $\mathcal N$.

Geodesics of $\mathcal N$ can be described in explicit terms. 
Let $\omega$ be the Maurer-Cartan form of $\mathcal N$, \emph{i.e.} $\omega: {\rm T}\mathcal N \to \mathfrak n$ is a $1$-form constant on left-invariant vector fields. Geodesics $\gamma:I\to \mathcal N$ are smooth curves such that $\gamma^*\omega$ is constant.
Geodesics from $0\in\mathcal N$ are given by $\exp(tv)$ for any $v\in\mathfrak n$. If $\exp(p)\in \mathcal N$ is any other point then the geodesics issued from $p$ are given by left multiplication: $\exp(p)\exp(tv)$.

Note that in the coordinates of $\mathfrak n$ given by the exponential, every geodesic is a straight line: from $0$ it is clear since geodesics are given by $\exp(tv)$ corresponding to the coordinates $tv$. If $\exp(p)\in\mathcal N$ then the Campbell-Hausdorff formula gives
\begin{equation}
\exp(p)\exp(tv) = \exp\left(p+tv+\frac 12 [p,tv]\right)
\end{equation}
hence coordinates of geodesics issued from $p$ are given by $p+tv+\frac 12[p,tv]$ which is again a straight line. This is a fact restricted to the $2$-nilpotent groups, in general the Campell-Hausdorff formula gives a polynomial which is not affine.

We will denote by $\exp_{\exp(p)}(tv)$ or $\exp_p(tv)$ the product $\exp(p)\exp(tv)$ or $p\exp(tv)$ depending whereas $p\in\mathfrak n$ or $p\in\mathcal N$. Since geodesics are given by left translations, we can still write $\exp_x(v) = x+_\mathcal{N} v$ as in the Euclidean case.
We will be careful not  to write $\exp_x(tv)=x+tv$ since $tv$ could indicate the dilatation by a factor $t$ as in the group $\mathcal{N}$ and not linearly as in the tangent space $\mathfrak n$.

Similarity transformations preserve the geodesic structure: if $\rho\in{\rm Sim}(\mathcal N)$ and if $\gamma$ is  geodesic, then $\rho\gamma$ is again geodesic. This can be checked manually since the group ${\rm Sim}(\mathcal N)$ is well known.

The space $\mathcal N$ is locally convex. The open balls that will be defined in the next section provide examples of arbitrary small open convex subsets.

\par\vspace{\baselineskip}\noindent
If $K$ is any manifold with a similarity structure, then we can pullback the geodesic structure on $K$. 
To be more specific, a curve $\gamma:I\to K$ is a geodesic if and only if $D\circ \gamma$ is a geodesic of $\mathcal N$.
This construction implies the equivariance
\begin{equation}
D(\exp_x(v))  = \exp_{D(x)}({\rm d} D \cdot v) = D(x) + {\rm d} D(v). 
\end{equation}
The exponential map at $x$ in $K$ is defined on an open subset. By definition, we will say that this is  the \emph{visible} open subset, denoted by $V_x\subset {\rm T}_xK$. This open subset is non empty since $D$ is a local diffeomorphism.

We are now interested in the case $K=\tilde M$. The following construction enables to get  arguments about $V_x$.
Let $p\in \tilde M$ and $g\in \pi_1(M)$. Suppose that there exists $v$ such that $\exp_p(v) = g\cdot p$.
Since $V_p$ is an open neighborhood of $v$, and since $\exp$ is a local diffeomorphism, we can reduce $V_p$ to a neighborhood $W_p$ of $v$ on which $\exp$ is a diffeomorphism. Therefore, there exists an open neighborhood $U_p$ of the origin in ${\rm T}_p\tilde M$ such that the application $G$ defined on $U_p$ by 
\begin{equation}
G(u) = (\exp_p^{-1}\circ g\circ \exp_p)(u)
\end{equation}
is well defined and verifies $G(0)=v$ and $G(U_p)= W_p$.
Remark that the maps $g$ and $G$ are equivariant:
\begin{equation}
\exp_p(G(u))=g\cdot \exp_p(u).
\end{equation}
We are now interested to extend $G$ to the full tangent space ${\rm T}_p\tilde M$. To do so, we need to treat the difficulty of the exponential of a non-visible vector. This is done by looking through the developing map.

As we did before, we can set $\exp_x(w)=x+w$ for $x\in \tilde M$ and $w\in \mathfrak n$ by looking through the developing map ($x\in \tilde M$ is send to $z\in\mathcal N$, the tangent spaces are identified and any tangent space in $\mathcal N$ is identified to $\mathfrak n$ by the parallel transport). For convenience we will denote by $p+V_p$ the image $\exp_p(V_p)$ of the exponential of all visible vectors.
\begin{lemma}\label{lem-dev-exp}
For any $w_1,w_2\in{\rm T}_p\tilde M$, if $D(p) + {\rm d}D(w_1) = D(p) + {\rm d}D(w_2)$ then $w_1=w_2$. In particular $D$ restricted to $p+V_p$ is injective since $D(p+v)=D(p)+{\rm d}D(v)$.
\end{lemma}

\begin{proof}
If $D(p) + {\rm d}D(w_1) = D(p) + {\rm d}D(w_2)$ then by unicity of the geodesics in $\mathcal N$ we deduce ${\rm d}D(w_1) = {\rm d}D(w_2)$. Since $D$ is  a local diffeomorphism, this implies $w_1=w_2$.
\end{proof}

In $\mathcal N$ every exponential is well defined. By the preceding lemma, what happens in $V_p$ is not different from what happens in the developing map.
This is why we define 
\begin{equation}
G(u) = {\rm d}D^{-1}_{p}\circ \exp_{D(p)}^{-1} \circ \rho(g)\circ \exp_{D(p)} \circ {\rm d}D_p(u).
\end{equation}
With $v$ as before, it remains true that $G(0)=v$. 
Again, for any $w\in V_p\cap G^{-1}(V_p)$
\begin{equation}
\exp_p(G(w)) = g\cdot \exp_p(w).\label{eq-comp-G}
\end{equation}
For, recall that by the lemma the developing map  is injective on $p+V_p$. We conclude by the following computation.
\begin{align}
D(g\cdot \exp_p(w)) &= \rho(g)\circ D(\exp_p(w))\\
D(\exp_p(G(w))) &= \exp_{D(p)}({\rm d}D\cdot G(w))\\
&= \rho(g)\circ \exp_{D(p)}({\rm d}D_p\cdot w)\\
&= \rho(g)\circ D(\exp_p(w))
\end{align}

The following proposition is crucial to the study. We will recall the notations.
\begin{proposition}\label{prop-conv-n}
Let $p\in \tilde M$ and $v\in V_p$ such that there exists $g\in\pi_1(M)$ verifying $g\cdot p = p+v$. The map $G$ exponentially equivariant with $g$ is defined by
\begin{equation}
G(u) = {\rm d}D^{-1}_{p}\circ \exp_{D(p)}^{-1} \circ \rho(g)\circ \exp_{D(p)} \circ {\rm d}D_p(u).
\end{equation}

Suppose that $U$ is a convex  subset  of $V_p$ and contains $v=G(0)$. For any $w\in{\rm T}_p\tilde M$, if $G(w)\in U$ then $w\in V_p$.
\end{proposition}
\begin{proof}
Suppose that $U$ verifies the hypotheses. For $t<\epsilon$ small enough $tw$ is visible and $G(tw)$ is close to $v$ and is hence  visible. Therefore $tw\in V_p\cap G^{-1}(V_p)$. We know by equation \eqref{eq-comp-G} that
\begin{equation}
\exp_p(G(tw)) = g\cdot \exp_p(tw).
\end{equation}
Since $\rho(g)$ transforms geodesics into geodesics, it transforms $tw$ onto a geodesic from $G(0)$ to $G(w)$.

Since $G(0),G(w)\in U$ by assumption, the left member is well defined for any $t\in [0,1]$ because $U$ is convex. The right member is defined for $t$ in $[0,\epsilon[$. But since the left member is always defined $[0,\epsilon[$ must be closed in $[0,1]$, therefore equal to $[0,1]$. It follows that $w$ is visible.
\end{proof}

\par\vspace{\baselineskip}\noindent
We will now give the first argument of Fried's theorem's proof. The hypothesis to keep in mind is that the similarity structure does not give a covering onto for $D$. The following ideas of convexity properties can be related to the work of Carrière in \cite{Carriere}. Some parts are exposed in \cite{Miner}.

Let $C$ be an open subset of $\tilde M$. We will say that $C$ is a \emph{convex} subset if $D|_C$ is a diffeomorphism with convex image in $\mathcal{N}$. Convexity in $\mathcal{N}$ is the property of containing every geodesic segment.

\begin{lemma}\label{lem-conv-g}
Let $p\in \tilde M$ and $0\in C\subset V_p$ such that $\exp_p(C)$ is convex. Let $g\in \pi_1(M)$. Then $g\exp_p(C)$ is a convex subset containing $gp$.
\end{lemma}
\begin{proof}
Since $g$ transforms geodesics into geodesics, it sends a convex to a convex.
\end{proof}

Remark that by the following proposition, we furthermore have that $g\exp_p(C)\subset gp + V_{gp}$.

\begin{proposition}\label{prop-conv-gen}
We have the following properties.
\begin{enumerate}
\item If $p\in C$ with $C\subset \tilde M$ convex, then $C\subset p+V_p$.
\item If $C_1,C_2$ are convex in $\tilde M$ with a non empty intersection, then $D|_{C_1\cup C_2}$ is injective.
\item If $\tilde M$ contains $p$ such that $p+V_p$ is convex, then $p+V_p=\tilde M$.
\end{enumerate}
\end{proposition}

\begin{proof}
In order.
\begin{enumerate}
\item If $z\in D(C)$ then there exists $w\in \mathfrak n$ such that  $z=D(p)+w$. 
Let $v= ({\rm d}D)^{-1}(w)$. Then $p+tv$ is well defined and belongs to  $C$ for $t$ small enough. But $D(p)+{\rm d}D_p(tv)$ is well defined and in $D(C)$ for all $t\leq 1$. Hence $p+tv$ is well defined for $t=1$ by taking $D^{-1}(D(p)+{\rm d}D_p(tv))$.
\item Take $p\in C_1\cap C_2$ and $z=D(p)$. If $D(x_1)=D(x_2)$ for $x_1\in C_1$ and $x_2\in C_2$ then the geodesic segment from $z$ to $D(x_1)$ is the same that joins $z$ to $D(x_2)$, hence for the same direction. Hence $x_1=x_2$.
\item It suffices to show that $p+V_p$ is closed in $\tilde M$ since it is already open and non empty. Let $y$ be in the adherence of $p+V_p$. Let $C\subset y+V_y$ be a convex subset containing $y$ (it can always be constructed since $D$ is a local diffeomorphism). 

Then $C\cap p+V_p$ has a non empty intersection, hence $D$ is injective on $C\cup p+V_p$. In $D(C\cup p+V_p)$ there exists $v$ such that $D(q) = D(p)+{\rm d}D_p(v)$. The geodesic $D^{-1}(D(p)+{\rm d}D_p(tv))$ is well defined and in $p+V_p$ for $t<1$ by hypothesis. At $t=1$ the point is well defined by injectivity of $D$, is equal to $q$ and belongs to the same geodesic, since for $t$ large enough it belongs to $C$. Therefore  $q\in p+ V_p$.
\qedhere
\end{enumerate}
\end{proof}
Remark that this last property shows that if $p+V_p$ is convex, then $D$ is a diffeomorphism. 

\begin{corollary}
Suppose that $D$ is not a covering map onto $\mathcal{N}$ (hence not a diffeomorphism), then $V_p$ is never equal to ${\rm T}_p\tilde M$.
\end{corollary}
\begin{proof}
The tangent space ${\rm T}_p\tilde M$ is convex since  ${\rm d}D_p$ is an isomorphism. If $V_p = {\rm T}_p\tilde M$, then $p+V_p=\tilde M$ and $D(p+V_p)=D(p)+\mathfrak n = \mathcal N$. It follows that $D:\tilde M \to \mathcal N$ is a diffeomorphism, hence a covering.
\end{proof}

\subsection{Proof of Fried's theorem}

We recall the hypotheses. The manifold $M$ is a closed $({\rm Sim}(\mathcal{N}),\mathcal{N})$-manifold such that the developing map $D:\tilde M\to \mathcal{N}$ is not a covering map.

Recall that we set a distance function $d_{\mathcal N}(x,y) = \|-x+y\|$ from the pseudo-norm $\|\cdot\|$ which is compatible with dilatations: $\|\lambda x\|=|\lambda|\|x\|$. The triangle inequality also remains true. This distance function is chosen so it is left-invariant: $d_{\mathcal N}(a+x,a+y)=d_{\mathcal N}(x,y)$. In particular, we can define the open ball of radius $R$ centered in $0$ to be 
\[ B(0,R) = \{x\in\mathcal N \,\mid\, d_{\mathcal N}(0,x)<R\}. \]
And in general, the open ball of radius $R$ centered in $p$ is given by the left translation of $B(0,R)$ to $p$.

The preceding corollary shows that for every $p\in \tilde M$, the image $D(p+V_p) =: D(\exp_p(V_p))$ is never equal to $\mathcal{N}$.
Hence, for each $p\in \tilde M$, there exists an open subset $B_p\subset V_p\subset {\rm T}_p\tilde M$ such that the image $D(p+B_p)$ is the maximal open ball in $D(p+V_p)$ centered in $D(p)$. 

We let
\begin{equation}
 r: \tilde M \to ]0,+\infty[
\end{equation}
be the map that associates to $p\in \tilde M$ the radius of the ball $D(p+B_p)$ in $\mathcal{N}$.

\begin{lemma}\label{lem-r-contract}
For $p\in \tilde M$ and $q\in p+B_p$, 
\begin{equation}
r(p)\leq r(q)+d_{\mathcal N}(D(p),D(q))
\end{equation}
and therefore $r$ is a local contraction.

Furthermore, if $g\in\pi_1(M)$ then $r(gp)=\lambda(g)r(p)$ with $\lambda(g)$ the dilatation factor of the holonomy transformation $\rho(g)\in{\rm Sim}(\mathcal{N})$.
\end{lemma}

\begin{proof}
Let $p\in \tilde M$. 
and let $q\in p+B_p$. 
By proposition \ref{prop-conv-gen}, $p+B_p\subset q+V_q$.
Let $v\in \partial B_q$ such that $q+v$ is not defined. Then $D(q)+{\rm d}D_q (v)$ does not belong to $D(q+V_q)$ hence does not belong to $D(p+B_p)$ which is precisely an open ball of radius $r(p)$. Therefore
\begin{align}
r(p) &\leq d_{\mathcal N}(D(p),D(q)+{\rm d}D_q(v)) \\
&\leq d_{\mathcal N}(D(q),D(q)+{\rm d}D_q(v)) + d_{\mathcal N}(D(p),D(q)) \\
&\leq  r(q)+d_{\mathcal N}(D(p),D(q)).
\end{align}

To prove the second part, we prove that $\rho(g)$ transforms $D(p+B_p)$ into $D(gp + B_{gp})$. If that is true then for $v\in\partial D(p+B_p)$, we have $\rho(g)v\in \partial D(gp+B_{gp})$ and therefore
\begin{align}
r(gp) = d_{\mathcal N}(D(gp),\rho(g)v) &= d_{\mathcal N}(\rho(g)D(p),\rho(g)v) \\
&= d_{\mathcal N}(0,\rho(g)(-D(p)+v))\\
&= \lambda(g) d_{\mathcal N}(0,-D(p)+v)\\
&=\lambda(g)d_{\mathcal N}(D(p),v)=\lambda(g)r(p).
\end{align}

In fact, it suffices to prove that $\rho(g) D(p+B_p)\subset D(gp+B_{gp})$ since  with $g^{-1}$ we would get $\rho(g)^{-1}D(gp+B_{gp}) =\rho(g^{-1})D(gp+B_{gp}) \subset D(p+B_p)$, and by applying $\rho(g)$ on both ends, we get $D(gp+B_{gp})\subset \rho(g)D(p+B_p)$.

By lemma \ref{lem-conv-g}, $g$ sends $p+B_p$ into a convex subset containing $gp$, and by proposition \ref{prop-conv-gen} this convex is contained in $gp+V_{gp}$. But $\rho(g)$ preserves open balls, hence $\rho(g)D(p+B_p)$ is an open ball contained in $D(gp+B_{gp})$ by maximality of $B_{gp}$.
\end{proof}

The equivariance between $r$ and $\lambda$ allows a sense of length in $M$ which will be invariant by the holonomy group.

In $\tilde M$ we set
\begin{equation}
d_{\tilde M}(p_1,p_2) =  \frac{d_\mathcal{N}(D(p_1),D(p_2))}{r(p_1)+r(p_2)}
\end{equation}
which is $\pi_1(M)$-invariant.

Let $p\in \tilde M$ and let $\epsilon>0$. We will describe open balls of radius $\epsilon$ in $\tilde M$ by locally looking at the pseudo-distance function $d_{\tilde M}$ on couples in $(p,p+B_p)$. On this set, $D$ is injective, hence $d_{\mathcal N}(D(p),D(x))$ really is a distance function. Also, by lemma \ref{lem-r-contract}, the function $r$ is contracting. This gives
\begin{align}
d_{\tilde M}(p,x) &= \frac{d_{\mathcal N}(D(p),D(x))}{r(p)+r(x)} \\
d_{\tilde M}(p,x) &\geq  \frac{d_{\mathcal N}(D(p),D(x))}{2r(x) + d_{\mathcal N}(D(p),D(x))}
\end{align}
hence if we suppose $d_{\tilde M}(p,x)<\epsilon$ with $\epsilon$ small enough, we get
\begin{align}
\frac{d_{\mathcal N}(D(p),D(x))}{2r(x) + d_{\mathcal N}(D(p),D(x))} &< \epsilon \\
\frac{d_{\mathcal N}(D(p),D(x))}{r(x)} &< \frac{2\epsilon}{1-\epsilon} .
\end{align}
If $r(p)\geq r(x)$ then the same inequality is true for $r(p)$ instead of $r(x)$. If $r(p)< r(x)$, then $p\in x+B_x$ ($p$ is visible from $x$ and lies inside the ball since it is closer to $x$ than the boundary  at distance $r(x)$) and by repeting the same argument for $(x,p)$ we get the preceding inequality with $r(p)$ instead of $r(x)$. In either cases
\begin{equation}
\frac{d_{\mathcal N}(D(p),D(x))}{r(p)} < \frac{2\epsilon}{1-\epsilon} .
\end{equation}

Conversely by using $r(x)\geq r(p)-d_{\mathcal N}(D(p),D(x))$
\begin{equation}
d_{\tilde M}(p,x) \leq \frac{d_{\mathcal N}(D(p),D(x))}{2r(p) - d_{\mathcal N}(D(p),D(x))}
\end{equation}
hence for $\epsilon>0$
\begin{align}
\frac{d_{\mathcal N}(D(p),D(x))}{2r(p) - d_{\mathcal N}(D(p),D(x))} &< \epsilon \\
\frac{d_{\mathcal N}(D(p),D(x))}{r(p)} &< \frac{2\epsilon}{1+\epsilon} 
\end{align}

This shows that for $\epsilon$ small enough, the ball
\begin{equation}
B_{\tilde M}(p,\epsilon) := \{x\in p+B_p \,\mid\, d_{\tilde M}(p,x)<\epsilon\}
\end{equation}
has an approximation in its developing image:
\begin{equation}
\left\{ \frac{d_{\mathcal N}(D(p),D(x))}{r(p)} < \frac{2\epsilon}{1+\epsilon} \right\} \subset D(B_{\tilde M}(p,\epsilon)) \subset \left\{ \frac{d_{\mathcal N}(D(p),D(x))}{r(p)}<\frac{2\epsilon}{1-\epsilon} \right\}
\end{equation}
is contained in an open of $\tilde M$. Hence those open balls provide the same basis for the topology of $\tilde M$.
This means that  $d_{\tilde M}(p,x)<\epsilon$ is true when in the ball $D(p+B_p)$ normalized by the radius $r(p)$, the distance between $x$ and $p$  is less than $\simeq 2\epsilon$.

If $g\in \pi_1(M)$, then by lemma \ref{lem-conv-g} and proposition \ref{prop-conv-gen} (compare also with the proof of lemma \ref{lem-r-contract} where we proved that $g(p+B_p) = gp+B_{gp}$), $gB_{\tilde M}(p,\epsilon)$ is a subset of $gp+B_{gp}$. The distance function $d_{\tilde M}$ being $\pi_1$-invariant, this shows that 
\begin{equation}
\forall g\in\pi_1(M), \; gB_{\tilde M}(p,\epsilon)= B_{\tilde M}(gp,\epsilon).
\end{equation}

In $M$, we can define a system of open neighborhoods by projecting the previously constructed balls of $\tilde M$.
\begin{equation}
B_M(x,\epsilon) := \{ \pi(B_{\tilde M}(p,\epsilon)) \,\mid\, p\in\pi^{-1}(x)\} = \pi(B_{\tilde M}(p,\epsilon)), p\in \pi^{-1}(x).
\end{equation}
The last equation being justified by the equality $gB_{\tilde M}(p,\epsilon) = B_{\tilde M}(gp,\epsilon)$.
For $\epsilon$ small enough, the ball $B_M(x,\epsilon)$ is therefore a trivializing neighborhood of $x$, and this system of open balls gives the same topology for $M$ as the original one.

\par\vspace{\baselineskip}\noindent
We will now construct holonomy transformations which will be very contracting, with a common center point and with no rotation. The idea is to take $v\in\partial B_p$ such that $p+v$ is not defined and to compare with $D(p)+{\rm d}D v$ in $\mathcal{N}$ where it must be defined. The holonomy transformations will be centered in $D(p)+{\rm d}Dv=: z+w$. See figure \ref{fig-fried} for the global setting.

\begin{figure}[ht]
\centering
\includegraphics[scale=0.75]{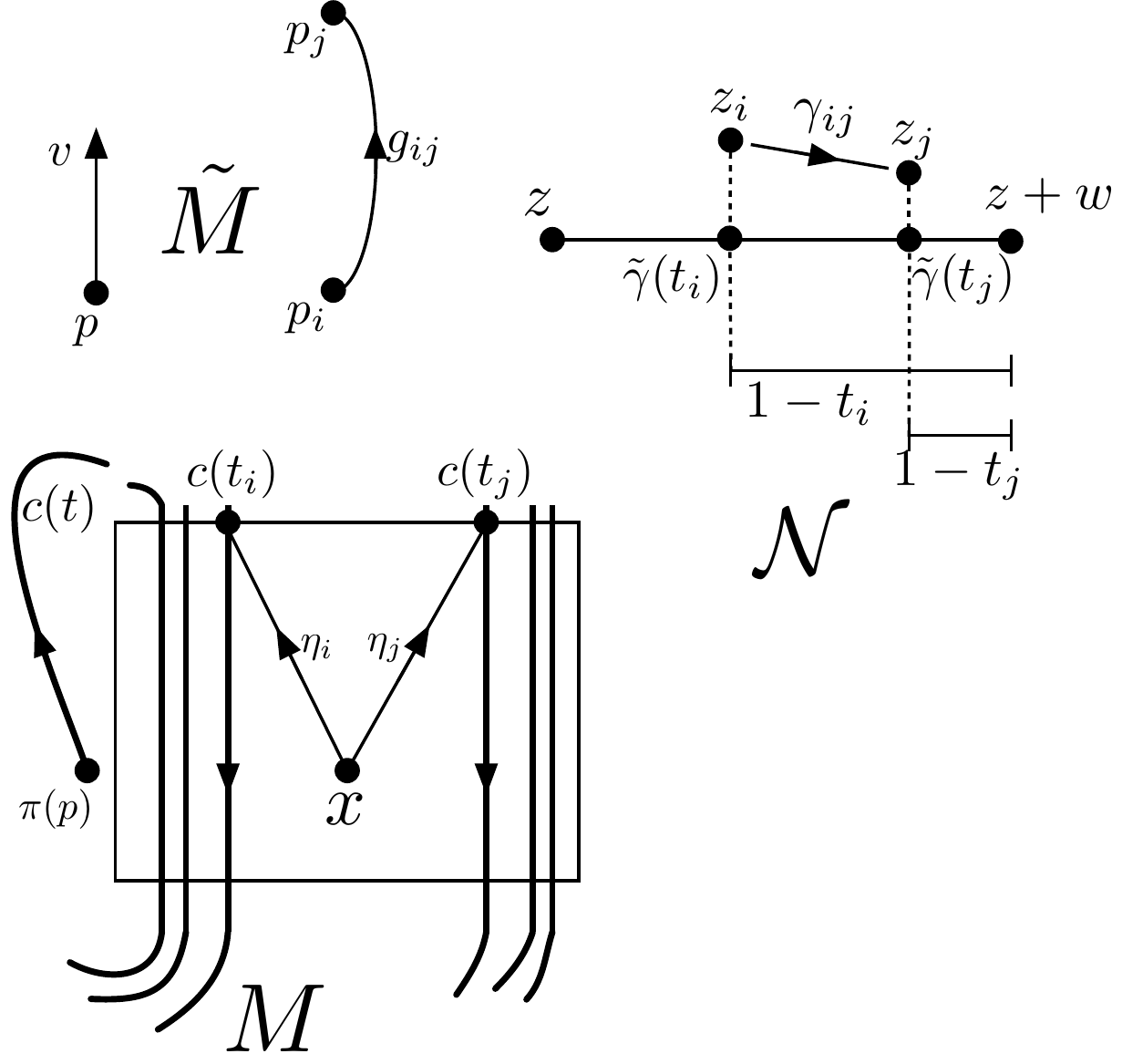}
\caption{The general setting.}\label{fig-fried}
\end{figure}

Consider $p\in \tilde M$ such that $\exp_p(tv)$ is defined for $0\leq t<1$ but not for $t=1$. The geodesic curve $[c(t)]=\exp_p(tv)$ is an incomplete geodesic. In $M$, the corresponding curve $c(t)=\pi([c(t)])$ is then an infinite long curve in a compact space. Hence, there is a recurrent point $x\in M$.

Let $B_M(x,\epsilon)$ be a ball with $\epsilon$ radius, for $\epsilon>0$ small enough such that $B_M(x,\epsilon)$ is trivializing $\pi:\tilde M \to M$. Let $0<t_1<\dots<t_n<\dots$ be entry times such that $t_n\to 1$ and $c(t_n)\in B_M(x,\epsilon)$ but $c([t_n,t_{n+1}])\not\subset B_M(x,\epsilon)$ (it just says that $c$ exits $B_M(x,\epsilon)$ before time $t_{n+1}$).
Since $\epsilon$ is small enough, for each $t_n$, up to homotopy we can set uniquely  $\eta_n$ the  segment from $x$ to $c(t_n)$ contained in $B_M(x,\epsilon)$.

By construction we have the following lemma.
\begin{lemma}\label{lem-technique}
For any $i$, 
 $[c(t_i)]\in p_i+B_{p_i}$ and  $d_{\tilde M}([c(t_i)],p_i)<\epsilon$.
\end{lemma}
\begin{proof}
By hypothesis and by the preceding discussion, since $B_M(x,\epsilon)$ is trivializing, if $c(t_i)\in B_M(x,\epsilon)$ then any lift $c(t_i)$ is  in $B_{\tilde M}(\hat p,\epsilon)$ for $\hat p\in\pi^{-1}(x)$. In particular  $[c(t_i)]$ is at distance at most $\epsilon$ from $p_i$ and lies in $p_i+B_{p_i}$ by definition of $B_{\tilde M}(p_i,\epsilon)$.
\end{proof}

We set
\begin{equation}
 g_{ij} =  \eta_j^{-1} \star c(t)|_{[t_i,t_{j}]} \star \eta_i,
\end{equation}
this is a family of transformations belonging to $\pi_1(M,x)$. 
We are now interested in $g_{ij}$ acting on $\tilde M$. 
The path $\tilde g_{ij}$ lifting $g_{ij}$ sends $p_i$ to $p_j$ by construction. We denote by $\gamma_{ij}$ the holonomy transformation  $\rho(g_{ij})$, and we denote by $\tilde \gamma_{ij}$ the image   $D(\tilde g_{ij})$ and by $\tilde \gamma$ the image $D([c(t)])$.

The vector $v$ initially chosen is sent to $w$ by ${\rm d}D$ and $z:=D(p)$. Each $p_i$ is sent to $z_i$ by $D$. See again figure \ref{fig-fried}.

\begin{proposition}
For $i,j$ large enough with $j\gg i$, the transformation $\gamma_{ij}$ is centered as close to $z+w$ as desired, with a dilatation factor as close to $0$ as desired and with an orthogonal part as close to identity as desired.
(In other terms, $\gamma_{ij}(x) = \lambda P(x-\beta)+\beta$ with $\lambda\to 0$, $P\to E$ and $\beta\to z+w$.)
\end{proposition}

The idea of the proof consists in taking a closer look to figure \ref{fig-fried}. Suppose for a minute that every $z_i$ is in fact $\tilde \gamma(t_i)$. Then the transformation $\gamma_{ij}$ sends $\tilde \gamma(t_i)$ to $\tilde \gamma(t_j)$.
 It is then clear that the transformation is very contracting and is centered in $z+w$ with no rotation.
It is the object of the proof to show that this approximation is correct.
Note that in practice,  a dilatation $\lambda(x,y) = (\lambda x,\lambda^2 y)$ does not stabilize a geodesic.

\begin{proof}
Assume that $\gamma_{ij}$ is an affine transformation given by
\begin{equation}
\gamma_{ij}(x) = \lambda P(x-\beta) + \beta,
\end{equation}
with $\lambda\in\mathbf{R}, P\in {\rm U}(n-1)$ and $\beta\in\mathcal{N}$  constant. We start by showing that $\beta\to z+w$ if $P=E_n$ and if $\lambda\to 0$.

Denote $(a_i,b_i)=\tilde\gamma(t_i)$ and $(c,d)=\beta$.
By construction,  $\gamma_{ij}$ sends $D(p_i)=z_i$ to $D(p_j)=z_j$.
We will later show   that $z_i$ and $\tilde\gamma(t_i)$ are very close for $t_i$ large enough.

Hence, the additional and final hypothesis is that 
 $\gamma_{ij}$ sends $(a_i,b_i)$ to $(a_j,b_j)$. It gives the following computation.
\begin{align}
(a_j,b_j) &= \lambda ( (a_i,b_i) - \beta) + \beta \\
&= \lambda \left( (a_i-c,b_i-d -{\rm Im}(a_i^*c)) \right) + (c,d) \\
&= \left( \lambda  (a_i-c) + c , \lambda^2(b_i-d -{\rm Im}(a_i^*c)) + d + {\rm Im}\left( \lambda (a_i-c) + c\right)^* c\right) \\
&= \left( (1-\lambda)c + \lambda  a_i, (1-\lambda^2)d + \lambda^2(b_i -{\rm Im}(a_i^*c)) + {\rm Im}\left(\lambda ((a_i-c)^*c\right)\right)
\end{align}
It then suffices to resolve the values of $c$ and $d$.
\begin{align}
a_j &= (1-\lambda)(c) + \lambda a_i\\
c&=(1-\lambda)^{-1}(a_j-\lambda a_i), \\
b_j &=(1-\lambda^2)d + \lambda^2(b_i -{\rm Im}(a_i^*c)) + {\rm Im}\left(\lambda (a_i-c)^*c)\right) \\
d &= (1-\lambda^2)^{-1}\left( b_j - \lambda^2(b_i -{\rm Im}(a_i^*c)) - \lambda{\rm Im}((a_i-c)^*c) \right)
\end{align}
We now use the fact that $(a_i,b_i)$ is given by
\begin{equation}
\exp_z(t_i(a_0,b_0)) = z + (t_ia_0,t_ib_0) = (z_1+t_ia_0, z_2 + t_ib_0 + t_i{\rm Im}(z_1^*b_0)),
\end{equation}
where $w=(a_0,b_0)$ and $z=(z_1,z_2)$.
We now use the hypothesis $\lambda\to 0$, we get $c\to a_j$ and $b\to b_j$ which show that for $t_i,t_j\to 1$ large enough, the point $\beta=(c,d)$ is as close as desired to $z+w$.

\par\vspace{\baselineskip}\noindent
It remains to show that we can indeed suppose that $P=E_n$, $\lambda \to 0$ and that $z_i$ is arbitrary close to $\tilde \gamma(t_i)$.

Since ${\rm U}(\mathcal N)$ is compact and since $\gamma_{jk}\gamma_{ij}=\gamma_{ik}$, the transformation $P$ of $\gamma_{ij}$ accumulates to the identity $E_n$.

To show $\lambda\to 0$, we will use lemma \ref{lem-technique}. First we have by lemma \ref{lem-technique} and \ref{lem-r-contract} and by definition of $[c(t)]$
\begin{equation}
r(p_j)  \leq r([c(t_j)]) + d_{\mathcal N}(z_j,\tilde\gamma(t_j))
\end{equation}
and by lemma \ref{lem-technique} and definition of $d_{\tilde M}$
\begin{equation}
d_\mathcal{N}(z_j,\tilde\gamma(t_j)) < \epsilon \left( r(p_j) + r([c(t_j)]) \right).
\end{equation}
This gives
\begin{equation}
r(p_j) \leq (1+\epsilon) r([c(t_j)]) + \epsilon r(p_j)
\end{equation}
Remark $r(p_j) = r(g_{ij}(p_i)) = \lambda(g_{ij}) r(p_i)$, it gives 
\begin{align}
\lambda(g_{ij}) r(p_i) &\leq  (1+\epsilon) r([c(t_j)]) + \epsilon\lambda(g_{ij}) r(p_i)\\
\lambda(g_{ij}) &\leq \frac{1+\epsilon}{1-\epsilon} \frac{r([c(t_j)])}{r(p_i)}.
\end{align}
Since the numerator tends to $0$, for $i$ fixed we get that $\lambda\to 0$ for $j$ tending to $+\infty$. So this is true for $i,j$ large enough such that $j\gg i$.
\end{proof}

We now recall the construction made for proposition \ref{prop-conv-n}.
\emph{Let $p\in \tilde M$. If for $g\in\pi_1(M)$, $g\cdot p$ is visible from $p$ by a vector in $B_p$, then if a vector $u$ is such that $G(u)\in B_p$, then in fact $p+u$ is visible from $p$, with $G$ given by}
\begin{equation}
G(u) = {\rm d}D^{-1}_{p} (-z + \rho(g)(D(p) +  {\rm d}D_p(u))).
\end{equation}
This allows to extend the set of vectors that can be seen from $p$ further than $B_p$. 
By passing through the developing map (by lemma \ref{lem-dev-exp}), if ${\rm d}D_pG(u)\in {\rm d}D_p\cdot B_p$ then it is true that $G(u)\in B_p$. Therefore if $D(p+G(u))\in D(p+ B_p)$ then $u$ is visible.

We now use the proposition.
We first need to check that $g_{ij}p$ is visible from $p$. We are only interested in the transformations $g_{ij}$ for large $i$ and $j$.
We know that $\gamma_{ij}z$ lies in the image of the ball $D(p+B_p)$, say with a vector $w'$ such that $\gamma_{ij}z=z+w'$. Say $v'$ corresponds to $w'$ by $({\rm d}D)^{-1}$. It follows that  $g_{ij}p$ is visible from $p$ by $v'\in B_p$.

The next proposition is inferred from this discussion.
\begin{proposition}\label{fried-demiespace}
The exponential based in $p$ is well defined on a “half-space” of ${\rm T}_p\tilde M$ given by
\begin{equation}
H_p =  \bigcup_{j \gg i \gg 0} {\rm d} D^{-1}(\gamma_{ij}^{-1}({\rm d} D \cdot  B_p)).\qed
\end{equation}
\end{proposition}

Remark that in the Euclidean case $\mathbf{F}=\mathbf{R}$ as in Fried's proof, this half-space is given by $\langle u,v\rangle <1$.
We now take a closer look to the form of $H_p$ in the general $\mathcal{N}$ to prove that it is in fact a half-space.

For, we look at the images $\gamma_{ij}^{-1}({\rm d} D\cdot B_p)$. For more clarity, we will suppose that ${\rm d}D\cdot B_p$ is the unit ball centered in $0$ in $\mathcal{N}$ (after taking the exponential based in $0$), denoted by $B$. In other words, $(u,I)\in B$ if and only if $\|u\|^2+\|I\| < 1$.

\begin{lemma}\label{lem-formI}
The boundary $\partial H_p\subset {\rm T}_p\tilde M$ is through the developing map an affine subspace either vertical (meaning that $(u,I)\in A\times {\rm Im}(\mathbf{F})$ with $A$ affine in $\mathbf{F}^{n-2}$) or horizontal (meaning that $(u,I)\in \mathbf{F}^{n-2}\times A$ with $A$ affine in $\mathbf{F}^{n-2}$).
\end{lemma}

\begin{proof}
From what we know, $g_{ij}$ is of the form
\begin{equation}
g_{ij}(x) = \lambda_{ij} (x - (z+w)) + (z+w)
\end{equation}
with $\lambda_{ij}\to 0$ and a small rotation that had been ignored.
Therefore, we look at the transformations
\begin{equation}
 g_\beta(\lambda)(x) =  \lambda (x-\beta) + \beta
\end{equation}
with $\beta\in\partial B$ the center and $\lambda>0$ the dilatation factor.

The image set for $x$ fixed is 
\begin{equation}
\{\lambda (x-\beta)+\beta \, \mid\, \lambda>0\}
\end{equation}
and this is half a parabola (sometimes degenerated into a straight line). We will examine the intersection of this parabola with $B$ for smalls $\lambda$. If this intersection is non empty then $x$ will be visible from $p$. 
Taking coordinates $x=(x_1,x_2)$ and $\beta=(\beta_1,\beta_2)$,
\begin{align}
g_\beta(\lambda)(x) &= \left(\lambda (x_1-\beta_1), \lambda^2(x_2-\beta_2 + {\rm Im}(x_1^*\beta_1))\right) + (\beta_1,\beta_2) \\
&= \left(\lambda(x_1-\beta_1) + \beta_1, \lambda^2(x_2-\beta_2)  + \beta_2 + \lambda(1+\lambda) {\rm Im}(x_1^*\beta_1)\right).
\end{align}

First, we examine the case $\beta_1\neq 0$. Up to applying a rotation of the subgroup $M$, we can suppose that $\beta_1 = (c,0\dots,0)$ with $c>0$ real.
Denoting $x_1 = (x_1^1,x_1^2,\dots,x_1^{n-1})$, we get
\begin{align}
g_\beta(\lambda)(x) &= \left(  \lambda(x_1-\beta_1) +\beta_1 , \beta_2 - \lambda c{\rm Im}(x_1^1) + \lambda^2(x_2-\beta_2 - c{\rm Im}(x_1^1)) \right) \\
&= \left( \lambda(x_1-\beta_1) + \beta_1, \beta_2 - \lambda c{\rm Im}(x_1^1) + o(\lambda)\right)
\end{align}
When $\lambda\to 0$, the condition $g_\beta(\lambda)(x)\in B$ does not depend on the coordinate $x_2$. Hence $\partial H_p$ is vertical (the condition is affine since it only involves ${\rm Im}(x_1^1)$).

Now we suppose $\beta_1=0$. Hence $\beta_2=\xi$ is unitary. The coordinates of the parabola then are
\begin{equation}
g_\beta(\lambda)(x) = (\lambda x_1, \lambda^2(x_2-\xi) + \xi).
\end{equation}
The norm of the first coordinate tends to $0$ and the norm of the second is smaller than $1$ if ${\rm Re}((x_2-\xi)^*\xi)<0$, \emph{i.e.} ${\rm Re}(x_2^*\xi)<1$. Therefore $\partial H_p$ is horizontal since this condition does not depend on $x_1$ and is affine in $x_2$.
\end{proof}

\begin{corollary}
For each $p\in \tilde M$, $H_p$ is convex.
\end{corollary}
\begin{proof}
Indeed, ${\rm d}D_p\partial H_p$ is affine and of real codimension 1, hence it separates $\mathfrak n$ in two connected components, each convex. One of them is ${\rm d}D_pH_p$.
\end{proof}

We now divide $\partial H_p$ onto $W_p\sqcup I_p$, where $W_p$ denotes the visible vectors and $I_p$ the invisible vectors of $\partial H_p$.

\begin{lemma}
The image $D(p)+{\rm d}D_p(I_p)$ is locally constant (hence constant) following $p$, this image is denoted $I$. Furthermore, $I$ is affine.
\end{lemma}
\begin{proof}
If $W_p=\emptyset$, then $I_p=\partial H_p$ is affine and its image by the developing map must be constant since $p+V_p = \tilde M$ is convex (because $H_p$ is convex and equal to $V_p$ if $\partial H_p$ is only constituted of invisible vectors, see proposition \ref{prop-conv-gen}).

Let $v_p\in I_p$. Suppose that $u\in W_p$. Let $q = p+u$. 
The point $q$ has a half-space $H_q$. We show that $D(p)+{\rm d}D_p(v_p)$  belongs to $D(q)+{\rm d}D_q(\partial H_q)$. This shows that $D(p)+{\rm d}D_p(I_q)$ is contained in the intersection of $D(p) + {\rm d}D_p(\partial H_p)$ with $D(q)+{\rm d}D_q(\partial H_q)$. Since $D(q)+{\rm d}D_q(\partial H_q)\subsetneq D(p)+{\rm d}D_p(\partial H_p)$, such an intersection decreases the topological dimension. By repeating the argument for a new $u$, the image of $I_p$  becomes constant following $p$ and affine. 

Since $p+H_p$ and $q+H_q$ are convex and with non empty intersection, it follows that $D$ is injective on $(p+H_p)\cup (q+H_q)$. Suppose that $D(p)+{\rm d}D_p(v_p)$ lies in $D(q+H_q)$, then this is locally true. Therefore, $D(c(t))=D(p+tv_p)$ is contained in $D(q+H_q)$ for $t\in ]T,1[$. By injectivity of $D$, this shows that $c(t)=p+tv_p$ is contained in $q+H_q$ for $t\in ]T,1[$ and this geodesic is defined for $t=1$ by hypothesis. But this contradicts that $c(1)$ is not defined. This shows that $D(p)+{\rm d}D_p(v_p)$ does not belong to $D(q+H_q)$.

Take $p'$ in $p+H_p$ such that $D(p)+{\rm d}D_p(v_p)$ belongs to $\partial D(p'+B_{p'})$. Then there exists $\rho(g_{ij})$ centered in $D(p)+{\rm d}D_p(v_p)$ very contracting with almost no rotation for $i,j$ large enough.

If $D(p)+{\rm d}D_p(v_p)$ does not belong to $D(q)+{\rm d}D_q(\partial H_q)$, then this last set  has no fixed point under $\rho(g_{ij})$. Therefore, $D(g_{ij}(q+H_q))$ is convex and contains $D(q)+{\rm d}D_q(\overline{B_q})$. But $g_{ij}(q+H_q)$ and $q+H_q$ intersect, hence $D$ is injective on the union.
If $c(t)\in q+B_q$ is a geodesic such that $c(0)=q$ and $c(1)$ is not defined, then this shows that $c(t)$ is well defined in $t=1$, since it is in $g_{ij}(q+H_q)$, absurd.
\end{proof}

Now the end of Fried's theorem's proof consists in showing that the holonomy group $\Gamma$ is discrete (compare with \ref{prop-j-group} and with \cite{Matsumoto}).

Since $I$ is constant, it follows that $D(\tilde M)$ does not intersect $I$ (since the exponential of a point is always locally defined and $I$ is the boundary of every $H_p$). Since $I$ is non empty by construction, by the proposition \ref{prop-deuxpoints}, it follows that the developing map is a covering onto its image, which is $\mathcal{N}-I$.

The main argument is the following lemma, which proof can already be found in the proof of \ref{prop-deuxpoints}.

\begin{lemma}
If there exist $f,g\in {\rm Sim}(N)$ such that $\lambda(f)\neq 1$ and the fixed point of $f$ is not fixed by $g$, then $\langle f,g\rangle$ is not a discrete subgroup of ${\rm Sim}(\mathcal{N})$.
\end{lemma}

Such applications are given by the various $g_{ij}$ by changing the base point $p$. If $I$ is not a single point, then two such maps $f,g$ exist. We will show that it does not occur.
If $\mathcal{N}-I$ is simply connected, then the holonomy group must be discrete, which contradicts the preceding lemma.

So we suppose, that $\mathcal{N}-I$ is not simply connected. This case occurs if $I$ is an affine subspace of real codimension 2. Let $H$ be an affine half-space with $\partial H=I$. By turning $H$ around $I$, we find that the universal cover of $\mathcal{N}-I$ is $H\times\mathbf{R}$.

Since $I$ is invariant by the holonomy group $\Gamma$, and since $\Gamma$ contains contractions (such as $g_{ij}$) it follows that $I$ is stable by dilatations.
The subgroup of ${\rm Sim}(\mathcal{N})$ stabilizing $I$ contains dilatations $f,g$ with different fixed points. The lifted group ${\rm Sim}(I)\times \mathbf{R}$ contains $\tilde f,\tilde g$ (essentially $\tilde f \simeq f\times \{0\}$ since $f= g_{ij}$ does not rotate much around $I$) again contracting and with different fixed points.

We now have a lifting of $(G,X)$-structures
\begin{equation}
({\rm Sim}(I)\times \mathbf{R},H\times \mathbf{R})\to ({\rm Sim}(I),\mathcal{N}-I) 
\end{equation}
and $M$ gets a lifted $({\rm Sim}(I)\times \mathbf{R},H\times \mathbf{R})$-structure. The developing map is given by the choice of a point: take $p\in \tilde M$ and $D(p)\in \mathcal N-I$, we have to choose $q\in H\times \mathbf R$ such that $q$ is send to $D(p)$ by the covering $H\times \mathbf R \to \mathcal N-I$, denote $q$ by $D'(p)$. The new developing map $D'$ is now fully prescribed by $D$ and $D'(p)$. Again, $D'$ is a covering map since $D$ is a covering map. For, take $q(t)$ a path in $H\times \mathbf R$, based in $q(0)=D'(p)$. It is sent to a path in $\mathcal N-I$ which is covered by a path $\tilde q(t)$ in $\tilde M$ since $D$ is a covering map, and $D'(\tilde q(t))=q(t)$ by local triviality.

Since $H\times \mathbf{R}$ is simply connected, the new holonomy group $\Gamma'$ must be discrete. But this is again contradicted by the preceding lemma  (slightly adapted).

\par\vspace{\baselineskip}\noindent
This concludes  Fried's theorem's proof since $I$ must be constituted of a single point (which must be totally fixed by $\Gamma$). Also by the proposition \ref{prop-deuxpoints} the rest of the theorem is shown.\qed

\section{Complete structures on closed manifolds}

The aim of this final section is to show the following theorem. This is more classical and not too difficult once Fried's theorem is given.

\begin{theorem}\label{thm}
Let $M$ be a closed $({\rm PU}(n,1),\partial \mathbf{H}^n)$-manifold. 
If $D$ is not surjective then it is a covering onto its image. Furthermore,
 $D$ is a covering on its image if, and only if, $D(\tilde M)$ is equal to a connected component of $\partial \mathbf{H}^n-L(\Gamma)$.
\end{theorem}

The proof of this theorem is cut into two parts.

\begin{proposition}\label{prop-nonsurj}
Let $M$ be a closed $({\rm PU}(n,1),\partial \mathbf{H}^n)$-manifold.
If the developing map is not surjective, then it is a covering map on its image.
\end{proposition}
\begin{proof}
Denote $\Omega = \partial\mathbf{H}^n-D(\tilde M)$.
\begin{itemize}
\item If $\Omega=\{a\}$, then up to conjugation, we can suppose that $\Omega=\{\infty\}$ and $D(\tilde M)\subset \mathcal{N}$. Since $\Omega$ is fixed by the holonomy group, we get a similarity structure $({\rm Sim}(\mathcal{N}),\mathcal{N})$ on $M$. Since $M$ is closed, it verifies the hypotheses of Fried's theorem and hence must be complete. Otherwise, another point would be missed by the developing map.
\item Suppose $\{a\}\subsetneq\Omega$. Then by minimality $L(\Gamma)\subset \Omega$.
If $L(\Gamma)=\emptyset$ then by lemma \ref{lem-gamma-eq}, the developing map is a covering onto the full $\partial\mathbf{H}^n$, absurd. Hence $L(\Gamma)$ consists of one or at least two points. If $L(\Gamma)=\{a\}$ then we can suppose it is $\infty$ and $D(\tilde M)\subset \mathcal{N}$. Again, by Fried's theorem and since $L(\Gamma)=\{\infty\}$, the developing map is in fact a covering onto $\mathcal{N}$, absurd. Hence, $L(\Gamma)$ consists of at least two points and is contained in $\Omega$. The conclusion follows from proposition \ref{prop-deuxpoints}.\qedhere
\end{itemize}
\end{proof}

\begin{proposition}\label{prop-complim}
Let $M$ be a closed $({\rm PU}(n,1),\partial \mathbf{H}^n)$-manifold.
The developing map is a covering on its image if, and only if, $D(\tilde M)$ is equal to a connected component of $\partial \mathbf{H}^n-L(\Gamma)$.
\end{proposition}

\begin{proof}
We start by supposing that $D(\tilde M)$ is equal to the full connected component of $\partial \mathbf{H}^n-L(\Gamma)$ on which $D(\tilde M)$ is defined.
\begin{itemize}
\item If $L(\Gamma)=\emptyset$ then by lemma \ref{lem-gamma-eq} the conclusion follows.
\item If $L(\Gamma)=\{a\}$ then $M$ gets a similarity structure which must be complete by Fried's theorem \ref{thm-fried}.
\item If $L(\Gamma)$ is constituted of at least two points, then by hypothesis the developing map avoids $L(\Gamma)$ and the proposition \ref{prop-deuxpoints} concludes.
\end{itemize}

Now we suppose that the developing map is a covering onto its image.
\begin{itemize}
\item If $D$ is a covering onto $\partial\mathbf{H}^n$ or $\partial\mathbf{H}^n-\{a\}$, then $D$ is a diffeomorphism and $\Gamma$ is discrete. It must avoid $L(\Gamma)$. If $D$ is a covering onto $\partial\mathbf{H}^n$ then $L(\Gamma)=\emptyset$. If $D$ is a covering onto $\partial\mathbf{H}^n-\{a\}$ then $L(\Gamma)\neq\emptyset$ by lemma \ref{lem-gamma-eq}, hence $L(\Gamma)$ is exactly $a$.
\item Now suppose that $D$ is a covering and avoids at least two points. Let $\Omega=\partial\mathbf{H}^n-D(\tilde M)$. Then by minimality $L(\Gamma)\subset \Omega$ and $L(\Gamma)$ can not be $\emptyset$ nor $\{a\}$, otherwise it would be a complete boundary manifold (by lemma \ref{lem-gamma-eq}) or a complete similarity manifold (by Fried's theorem \ref{thm-fried}) and $\Omega$ would be equal to $\emptyset$ or $a$. So $L(\Gamma)$ contains at least two points and we conclude by proposition \ref{prop-deuxpoints}.\qedhere
\end{itemize}
\end{proof}

\printbibliography

\end{document}